\documentclass{article}%
\usepackage{graphicx}
\usepackage{amsmath}
\usepackage{amsfonts}
\usepackage{amssymb}%
\providecommand{\U}[1]{\protect\rule{.1in}{.1in}}
\providecommand{\U}[1]{\protect\rule{.1in}{.1in}}
\providecommand{\U}[1]{\protect\rule{.1in}{.1in}}
\providecommand{\U}[1]{\protect\rule{.1in}{.1in}}
\providecommand{\U}[1]{\protect\rule{.1in}{.1in}}
\providecommand{\U}[1]{\protect\rule{.1in}{.1in}}
\providecommand{\U}[1]{\protect\rule{.1in}{.1in}}
\providecommand{\U}[1]{\protect\rule{.1in}{.1in}}
\providecommand{\U}[1]{\protect\rule{.1in}{.1in}}
\providecommand{\U}[1]{\protect\rule{.1in}{.1in}}
\providecommand{\U}[1]{\protect\rule{.1in}{.1in}}
\providecommand{\U}[1]{\protect\rule{.1in}{.1in}}
\providecommand{\U}[1]{\protect\rule{.1in}{.1in}}
\providecommand{\U}[1]{\protect\rule{.1in}{.1in}}
\providecommand{\U}[1]{\protect\rule{.1in}{.1in}}
\providecommand{\U}[1]{\protect\rule{.1in}{.1in}}
\providecommand{\U}[1]{\protect\rule{.1in}{.1in}}
\providecommand{\U}[1]{\protect\rule{.1in}{.1in}}
\providecommand{\U}[1]{\protect\rule{.1in}{.1in}}
\providecommand{\U}[1]{\protect\rule{.1in}{.1in}}
\providecommand{\U}[1]{\protect\rule{.1in}{.1in}}
\providecommand{\U}[1]{\protect\rule{.1in}{.1in}}
\providecommand{\U}[1]{\protect\rule{.1in}{.1in}}
\providecommand{\U}[1]{\protect\rule{.1in}{.1in}}
\providecommand{\U}[1]{\protect\rule{.1in}{.1in}}
\providecommand{\U}[1]{\protect\rule{.1in}{.1in}}
\providecommand{\U}[1]{\protect\rule{.1in}{.1in}}
\providecommand{\U}[1]{\protect\rule{.1in}{.1in}}
\providecommand{\U}[1]{\protect\rule{.1in}{.1in}}

\newtheorem{theorem}{Theorem}
{}

\newtheorem{corollary}{Corollary}

\newtheorem{definition}{Definition}

{}

\newtheorem{proposition}{Proposition}
\newtheorem{remark}{Remark}

\newtheorem{summary}{Summary}
\newenvironment{proof}[1][Proof]{\textbf{#1.} }{\ \rule{0.5em}{0.5em}}

\begin{document}

\title{On the finite-zone periodic PT-symmetric potentials}
\author{O. A. Veliev\\{\small \ Depart. of Math., Dogus University, }\\{\small Ac\i badem, 34722, Kadik\"{o}y, \ Istanbul, Turkey.}\\\ {\small e-mail: oveliev@dogus.edu.tr}}
\date{}
\maketitle

\begin{abstract}
In this paper we find explicit conditions on the periodic PT-symmetric
complex-valued potential $q$ for which the number of gaps in the real part of
the spectrum of the one-dimensional Schrodinger operator $L(q)$ is finite.



\end{abstract}

\section{ Introduction}

We study the PT-symmetric periodic locally integrable potential $q$ for which
the real part of the spectrum of the operator $L(q)$ generated in
$L_{2}(-\infty,\infty)$ by the expression
\begin{equation}
-y^{^{\prime\prime}}(x)+q(x)y(x)
\end{equation}
contains a half line $[R,\infty)$ for some number $R$. Without loss of
generality, we assume that the period of $q$ is $1$ and the integral of $q$
over $[0,1]$ is zero. Thus
\begin{equation}
q\in L_{1}[0,1],\text{ }%
{\textstyle\int\nolimits_{0}^{1}}
q(x)dx=0,\text{ }q(x+1)=q(x),\text{ }\overline{q(-x)}=q(x).
\end{equation}

It is well-known that [8, 10] the spectrum $\sigma(L)$ of $L$ is the union of
the spectra $\sigma(L_{t})$ of $L_{t}$ for $t\in(-\pi,\pi]$ generated in
$L_{2}[0,1]$ by (1) and the boundary conditions
\[
y(1)=e^{it}y(0),\text{ }y^{^{\prime}}(1)=e^{it}y^{^{\prime}}(0).
\]
The spectrum of $L_{t}$ consists of the eigenvalues called as Bloch
eigenvalues. The eigenvalues of $L_{0}$ and $L_{\pi}$ are called periodic and
antiperiodic eigenvalues respectively. Both of them are called $2$-periodic
eigenvalues. If $q$ is PT-symmetric, then the following implications hold%
\begin{equation}
\lambda\in\sigma(L_{t})\Longrightarrow\overline{\lambda}\in\sigma
(L_{t}),\text{ }\lambda\in\sigma(L)\Longrightarrow\overline{\lambda}\in
\sigma(L).
\end{equation}
The first and second implications were proved in [6] and [11] respectively. A
basic mathematical question of PT-symmetric quantum mechanics concerns the
reality of the spectrum of the considered Hamiltonian (see [2, 9 and
references of them]). In the first papers [1, 3, 5] about the PT-symmetric
periodic potential, the appearance and disappearance of real energy bands for
some complex-valued PT-symmetric periodic potentials under perturbations have
been reported. Shin [11] showed that the appearance and disappearance of such
real energy bands imply the existence of nonreal band spectra. He involved
some condition on the Hill discriminant to show the existence of the nonreal
curves in the spectrum. Caliceti and Graffi [4] found explicit condition on
the Fourier coefficients of the potential providing the nonreal spectra for
small potentials.

In [13], we first considered the general spectral property of $L$ and proved
that the main part of $\sigma(L)$ is real and contains the large part of
$[0,\infty).$ Using this we found necessary and sufficient condition on $q$
for finiteness of the number of the nonreal arcs in $\sigma(L(q))$. Finally,
we considered the connections between spectrality of $L$ and the reality of
$\sigma(L)$.

This paper can be considered as continuation of [13]. Here using some results
of [13] we consider the potentials $q$ for which the real part
$\operatorname{Re}\left(  \sigma(L(q))\right)  $ of $\sigma(L(q))$\ contains a
half line, which implies that the number of gaps in $\operatorname{Re}\left(
\sigma(L(q))\right)  $ is finite. If $q$ is real, such potentials are called
finite-zone potentials (see for example [7] and its references). Thus we
consider the finite-zone PT-symmetric periodic complex-valued potentials. We
do not discuss the real finite-zone potentials, since the results and
approaches of the self-adjoint case are not used here and they have no
connections with the results of this paper.

However, to see the diversity of the results for the PT-symmetric finite-zone
potentials and the real finite-zone potentials let us discuss $\sigma(L(q))$
in term of $2$-periodic eigenvalues. If $q$ is real, then it is well-known
that [7] all Bloch eigenvalues are real and $\sigma(L(q))$ consists of the
intervals
\[
\Gamma_{1}=:[\lambda_{0}(a),\lambda_{1}^{-}(a)],\text{ }\Gamma_{2}%
=:[\lambda_{1}^{+}(a),\lambda_{2}^{-}(a)],\text{ }\Gamma_{3}=:[\lambda_{2}%
^{+}(a),\lambda_{3}^{-}(a)],\text{ }\Gamma_{4}=:[\lambda_{3}^{+}%
(a),\lambda_{4}^{-}(a)],...,
\]
where \ $\lambda_{0}(a)<\lambda_{1}^{-}(a)\leq\lambda_{1}^{+}(a)<$
$\lambda_{2}^{-}(a)\leq\lambda_{2}^{+}(a)<$ $\lambda_{3}^{-}(a)\leq\lambda
_{3}^{+}(a)<....$ are $2$-periodic eigenvalues. The bands $\Gamma_{n}$ and
$\Gamma_{n+1},...$ of the spectrum $\sigma(L(q))$ are separated by the gaps
$(\lambda_{n}^{-},\lambda_{n}^{+})$ if and only if $\lambda_{n}^{-}%
<\lambda_{n}^{+}.$ The last inequality holds if and only if $\lambda_{n}^{-}$
is a simple eigenvalue. Thus, it is clear that, the real potential is
finite-zone if and only if all but a finite number of $2$-periodic eigenvalues
are double eigenvalues.

In this paper we prove that the PT-symmetric periodic complex-valued potential
$q$ is finite-zone if and only if all but a finite number of $2$-periodic
eigenvalues are either double eigenvalues (Case 1) or nonreal eigenvalues
(Case 2). \ Thus for the PT-symmetric complex-valued periodic potentials there
is additional case (Case 2) that happens for the large class of the complex
potentials. This give us a possibility to construct a large class of
PT-symmetric finite-zone potentials. In this paper we find the explicit
conditions on PT-symmetric periodic complex-valued potential $q$ for which the
number of gaps in $\operatorname{Re}\left(  \sigma(L(q))\right)  $ is finite.

\section{Main Results}

In this section we list the results of [12, 13], which will be used in this
paper, as summaries and then prove the main results by using the summaries.
The results of [12] about numerations and asymptotic behavior of the Bloch
eigenvalues can be summarized as follows.

\begin{summary}
$(a)$ The eigenvalues of $L_{t}$ can be numbered (counting multiplicity) by
elements of $\mathbb{Z}$, i.e. $\sigma(L_{t})=\left\{  \lambda_{n}%
(t):n\in\mathbb{Z}\right\}  $, such that for each $n$ the function
$\lambda_{n}$ is continuous on $[0,\pi]$ and $\lambda_{n}(-t)=\lambda_{n}(t).$
Thus $\sigma(L)$ is the union of $\Gamma_{n}$ for $n\in\mathbb{Z},$ where
$\Gamma_{n}=\left\{  \lambda_{n}(t):t\in\lbrack0,\pi]\right\}  $ is a
continuous curve called as a band of $\sigma(L)$. The eigenvalues $\lambda
_{n}(0)$ and $\lambda_{n}(\pi)$ for $n\in\mathbb{Z}$ are periodic and
antiperiodic eigenvalues respectively.

$(b)$ For $h=1/20\pi,$ there exists $N$ such that for $|n|\geq N$ the
followings hold.

If $t\in\lbrack h,\pi-h]$ then the eigenvalue $\lambda_{n}(t)$ is simple and
\begin{equation}
\left\vert \lambda_{n}(t)-(2\pi n+t)^{2}\right\vert \leq n^{-1/2},\text{
}\left\vert \lambda_{n}(t)-\lambda_{k}(t)\right\vert >2\pi^{2}n,\text{
}\forall k\neq n.
\end{equation}
If $t\in\lbrack0,h],$ then there exist $2$ eigenvalues (counting
multiplicity), denoted by $\lambda_{n}(t)$ and $\lambda_{-n}(t)$ such that
\begin{equation}
\left\vert \lambda_{\pm n}(t)-(2\pi n+t)^{2}\right\vert <n,\text{ }\left\vert
\lambda_{\pm n}(t)-\lambda_{k}(t)\right\vert >2\pi^{2}n,\text{ }\forall
k\neq\pm n.
\end{equation}
If $t\in\lbrack\pi-h,\pi],$ then there exist $2$ eigenvalues (counting
multiplicity), denoted by $\lambda_{n}(t)$ and $\lambda_{-n-1}(t),$ such that%
\begin{align}
\left\vert \lambda_{n}(t)-(2\pi n+t)^{2}\right\vert  &  <n,\text{ }\left\vert
\lambda_{-n-1}(t)-(2\pi n+t)^{2}\right\vert <n,\\
\left\vert \lambda_{n}(t)-\lambda_{k}(t)\right\vert  &  >2\pi^{2}n,\text{
}\left\vert \lambda_{-n-1}(t)-\lambda_{k}(t)\right\vert >2\pi^{2}n,\text{
}\forall k\neq n,-(n+1).\nonumber
\end{align}

\end{summary}

The first results of [13] about $\sigma(L(q))$ with potential (2) can be
summarized as follows.

\begin{summary}
$(a)$ For each $n\in\mathbb{Z}$ the band $\Gamma_{n}$ is a single open curve
with the end points $\lambda_{n}(0)$ and $\lambda_{n}(\pi)$. Two bands
$\Gamma_{n}$ and $\Gamma_{m}$ may have at most one common point.

$(b)$ If $\lambda_{n}(t_{1})$ and $\lambda_{n}(t_{2})$ are real numbers, where
$0\leq t_{1}<t_{2}\leq\pi$ then $\left\{  \lambda_{n}(t):t\in\lbrack
t_{1},t_{2}]\right\}  $ is an interval of the real line and subset of
$\operatorname{Re}\left(  \Gamma_{n}\right)  :=\Gamma_{n}\cap\mathbb{R}$.

$(c)$ For $t\in\lbrack h,\pi-h]$ any eigenvalue $\lambda_{n}(t)$ and for
$t\in\lbrack0,h)\cup(\pi-h,\pi]$ any double eigenvalue $\lambda_{n}(t)$ is a
real number if $\left\vert n\right\vert \geq N$, where $h$ and $N$ are defined
in Summary 1.
\end{summary}

The Summary 2$(b)$ immediately implies that $\operatorname{Re}\left(
\Gamma_{n}\right)  $ is a connected subset of $\mathbb{R}$. Therefore the
following statement is true.

\begin{proposition}
$\operatorname{Re}\left(  \Gamma_{n}\right)  $ for each $n\in\mathbb{Z}$ is
either empty set or a point or an interval.
\end{proposition}

Now using the notations and inequalities of Summary 1 we consider
$\operatorname{Re}\left(  \Gamma_{n}\right)  $ for $\left\vert n\right\vert
\geq N.$ For this introduce the following notations for $n\geq N:$
\begin{align}
a_{n}  &  =\inf\left\{  t\in\lbrack0,h]:\lambda_{n}(t)\in\mathbb{R}\right\}
,\text{ }b_{n}=\sup\left\{  t\in\lbrack\pi-h,\pi]:\lambda_{n}(t)\in
\mathbb{R}\right\}  ,\nonumber\\
b_{-n}  &  =\inf\left\{  t\in\lbrack0,h]:\lambda_{-n}(t)\in\mathbb{R}\right\}
,\text{ }a_{-n}=\sup\left\{  t\in\lbrack\pi-h,\pi]:\lambda_{-n}(t)\in
\mathbb{R}\right\}  ,\\
A_{n}  &  =\lambda_{n}(a_{n}),\text{ }B_{n}=\lambda_{n}(b_{n}).\nonumber
\end{align}
Since $\lambda_{\pm n}(h)$ and $\lambda_{\pm n}(\pi-h)$ are real number (see
Summary 2$(c))$ and $\lambda_{\pm n}$ is continuous on $[0,\pi]$ (see Summary
1$(a)$), $\inf$ and $\sup$ in (7) exist.

\begin{theorem}
If (2) holds, then the numbers $\lambda_{n}(a_{n})$ and $\lambda_{n}(b_{n})$
for $\left\vert n\right\vert \geq N$ are the real numbers and%
\begin{equation}
\lambda_{n}(a_{n})<\lambda_{n}(b_{n}),\text{ }\operatorname{Re}\left(
\Gamma_{n}\right)  =[A_{n},B_{n}].
\end{equation}

\end{theorem}

\begin{proof}
The reality of $\lambda_{n}(a_{n})$ and $\lambda_{n}(b_{n})$ follows from the
definitions of $a_{n}$ and $b_{n}$ and the continuity of $\lambda_{n}.$ The
inequality in (8) follows from (5) and (6). Therefore using Summary 2$(b)$ and
taking into account that $\Gamma_{n}$ is a single open curve (see Summary
2$(a)$) we get
\begin{equation}
\left\{  \lambda_{n}(t):t\in\lbrack a_{n},b_{n}]\right\}  =[A_{n}%
,B_{n}]\subset\operatorname{Re}\left(  \Gamma_{n}\right)  .
\end{equation}
On the other hand, it follows from the definition of $a_{n}$ and $b_{n}$ that
\begin{equation}
\left\{  \lambda_{n}(t):t\in\lbrack0,a_{n})\right\}  \cap\mathbb{R}%
=\varnothing,\text{ }\left\{  \lambda_{n}(t):t\in(b_{n},\pi]\right\}
\cap\mathbb{R}=\varnothing.
\end{equation}
Thus the equality in (8) follows from (9) and (10)
\end{proof}

\begin{theorem}
Suppose that $q$ satisfies (2). 

$(a)$ The following inequalities hold
\begin{equation}
A_{-n}<B_{-n}\leq A_{n}<B_{n}\leq A_{-n-1}<B_{-n-1}%
\end{equation}
for $n\geq N,$ where $N$ is defined in Summary 1. Moreover, $\sigma
(L(q))\cap\lbrack A_{-N},\infty)$ is the union of the intervals $[A_{k}%
,B_{k}]$ for $\left\vert k\right\vert \geq N.$

$(b)$ The gaps in $\sigma(L(q))\cap\lbrack A_{-N},\infty)$ are $(B_{-n}%
,A_{n})$ and $(B_{n},A_{-n-1})$ for $n\geq N$

if $B_{-n}<A_{n}$ and $B_{n}<A_{-n-1}$ \ respectively.

$(c)$ $\sigma(L)\cap\lbrack A_{-N},\infty)=[A_{-N},\infty)$ if and only if
$B_{-n}=A_{n}$ and $B_{n}=A_{-n-1}$ for $n\geq N.$
\end{theorem}

\begin{proof}
$(a)$ It readily follows from inequalities (5) and (6) that $A_{-n}<B_{n}$ for
$n\geq N.$ Therefore if $B_{-n}>A_{n}$ then the intervals $[A_{-n},B_{-n}]$
and $[A_{n},B_{n}]$ have common subinterval that contradicts Summary 2$(a)$.
Thus $B_{-n}\leq A_{n}$ and by (8) we have $A_{-n}<B_{-n}\leq A_{n}<B_{n}.$ In
the same way, we prove the other part of (11). Now (8) and (11) imply that
$\left(  \sigma(L(q))\cap\lbrack A_{-N},\infty)\right)  \supset\lbrack
A_{n},B_{n}]$ for $\left\vert n\right\vert \geq N$. On the other hand, it
follows from (4)-(6) that $\lambda_{n}(t)\notin\lbrack B_{-N},\infty)$ for
$t\in\left[  0,\pi\right]  $ and $\left\vert n\right\vert <N.$ Thus $(a)$ is
proved. The proofs of $(b)$ and $(c)$ follow from $(a)$
\end{proof}

Now we consider in detail the real parts of the bands $\Gamma_{n}$ of
$\sigma(L(q))$ with potential (2) by using the results of Theorems 4 and 5 of
[13] formulated in the suitable form as the following Summary.

\begin{summary}
Let $h$ and $N$ be the numbers defined in Summary 1 and $n\geq N.$ The
eigenvalue $\lambda_{n}(0)$ is a nonreal number if and only if there exists
$\varepsilon_{n}\in(0,h)$ such that
\begin{equation}
\left\{  \lambda_{n}(t):t\in\lbrack0,\varepsilon_{n})\right\}  \cap
\mathbb{R}=\varnothing,\text{ }\lambda_{n}(\varepsilon_{n})=\lambda
_{-n}(\varepsilon_{n})\in\mathbb{R}.
\end{equation}
The eigenvalue $\lambda_{n}(\pi)$ is a nonreal number if and only if there
exists $\delta_{n}\in(\pi-h,\pi)$ such that
\begin{equation}
\left\{  \lambda_{n}(t):t\in(\delta_{n},\pi]\right\}  \cap\mathbb{R}%
=\varnothing,\text{ }\lambda_{n}(\delta_{n})=\lambda_{-n-1}(\delta_{n}%
)\in\mathbb{R}.
\end{equation}

\end{summary}

The following proposition is a consequence of (7) and Summary 3.

\begin{proposition}
Let $N$ be the number defined in Summary 1 and $n\geq N.$

$(a)$ If $\lambda_{n}(0)$ ($\lambda_{n}(\pi))$ is a real number, then
$A_{n}=\lambda_{n}(0)$ ($B_{n}=\lambda_{n}(\pi).$

$(b)$ If $\lambda_{-n}(\pi)$ ($\lambda_{-n}(0))$ is a real number, then
$A_{-n}=\lambda_{-n}(\pi)$ ($B_{-n}=\lambda_{-n}(0))$.

$(c)$ If $\lambda_{n}(0)$ ($\lambda_{n}(\pi))$ is a nonreal number, then
$A_{n}=\lambda_{n}(\varepsilon_{n})$ ($B_{n}=\lambda_{n}(\delta_{n}).$

$(d)$ If $\lambda_{-n}(\pi)$ ($\lambda_{-n}(0))$ is a nonreal number, then
$A_{-n}=\lambda_{-n}(\delta_{n})$ ($B_{-n}=\lambda_{-n}(\varepsilon_{n}))$.
\end{proposition}

\begin{definition}
$\lambda_{n}(t)$ is called SR eigenvalue if it is simple eigenvalue and real number.
\end{definition}

\begin{theorem}
Suppose that (2) holds and $n\geq N,$ where $N$ is defined in Summary 1.

$(a)$ If $\lambda_{n}(0)$ ($\lambda_{n}(\pi)$) is a SR eigenvalue, then
$\lambda_{-n}(0)$ ($\lambda_{-n-1}(\pi))$ is also SR eigenvalue, $\lambda
_{-n}(0)<\lambda_{n}(0)$ ($\lambda_{n}(\pi)<\lambda_{-n-1}(\pi)$) and $\left(
\lambda_{-n}(0),\lambda_{n}(0)\right)  $ ($\left(  \lambda_{n}(\pi
),\lambda_{-n-1}(\pi)\right)  $) is the gap of $\sigma(L)\cap\lbrack
a_{-N},\infty)$ lying between $\operatorname{Re}\left(  \Gamma_{-n}\right)  $
and $\operatorname{Re}\left(  \Gamma_{n}\right)  $ ($\operatorname{Re}\left(
\Gamma_{n}\right)  $ and $\operatorname{Re}\left(  \Gamma_{-n-1}\right)  $).

$(b)$ If $\lambda_{n}(0)$ ($\lambda_{n}(\pi)$) is not a SR eigenvalue then
there is not a gap between $\operatorname{Re}\left(  \Gamma_{-n}\right)  $ and
$\operatorname{Re}\left(  \Gamma_{n}\right)  $ ($\operatorname{Re}\left(
\Gamma_{n}\right)  $ and $\operatorname{Re}\left(  \Gamma_{-n-1}\right)  $).

$(c)$ The potential $q$ is finite-zone if and only if the total number of the
SR periodic and antiperiodic eigenvalues is finite.
\end{theorem}

\begin{proof}
$(a)$\ Suppose that $\lambda_{n}(0)$ is a SR eigenvalue. By Summary 1$(b)$
(see (5)) there exist 2 eigenvalues counting multiplicity lying in
$n$-neighborhood of $(2\pi n)^{2}$ and these eigenvalues are denoted by
$\lambda_{-n}(0)$ and $\lambda_{n}(0).$ Therefore $\lambda_{-n}(0)$ is also a
simple eigenvalue. If $\lambda_{-n}(0)$ is a nonreal eigenvalue, then by (3),
$\overline{\lambda_{-n}(0)}$ is also periodic eigenvalue lying in
$n$-neighborhood of $(2\pi n)^{2}.$ Then we obtain that there exist $3$
eigenvalues lying in $n$-neighborhood of $(2\pi n)^{2}$ which contradicts
Summary 1$(b).$ Thus $\lambda_{-n}(0)$ is also SR eigenvalue. Therefore, by
Proposition 2$(a),$ we have $A_{n}=\lambda_{n}(0)$ and $B_{-n}=\lambda
_{-n}(0)$ and $A_{n}\neq B_{-n}.$ It with (11) and Theorem 2$(b)$ implies that
$B_{-n}<A_{n}$ and $\left(  \lambda_{-n}(0),\lambda_{n}(0)\right)  $ is the
gap in the spectrum. The proof for the case when $\lambda_{n}(\pi)$ is a SR
eigenvalue are the same.

$(b)$ If $\lambda_{n}(0)$ is not a SR eigenvalue then the following cases are possible.

Case 1. $\lambda_{n}(0)$ is a double eigenvalue. Then by Summary 1$(b)$ and
Summary 2$(c)$ $\lambda_{n}(0)=\lambda_{-n}(0)\in\mathbb{R}$. Thus by
Propositions 2 we have $A_{n}=\lambda_{n}(0)=\lambda_{-n}(0)=B_{-n}$ which
means that there is not a gap between $[A_{-n},B_{-n}]$ and $[A_{n},B_{n}]$.

Case 2. $\lambda_{n}(0)$ is a nonreal eigenvalue. Then using Proposition 2 and
(12) we again obtain that $B_{-n}=A_{n}$ which again means that there is not a
gap between $[A_{-n},B_{-n}]$ and $[A_{n},B_{n}]$. In the same way, instead of
(12) using (13), we prove that if $\lambda_{n}(\pi)$ is not a SR eigenvalue
then there is not a gap between $[A_{n},B_{n}]$ and $[A_{-n-1},B_{-n-1}]$.

$(c)$ If the total number of the SR periodic and antiperiodic eigenvalues is
finite then there exists $m$ such that $\lambda_{n}(0)$ ($\lambda_{n}(\pi)))$
for $\left\vert n\right\vert >m$ is not a SR eigenvalue and hence by $(b)$
there is no gap between $\operatorname{Re}\left(  \Gamma_{-n}\right)  $ and
$\operatorname{Re}\left(  \Gamma_{n}\right)  $ ($\operatorname{Re}\left(
\Gamma_{n}\right)  $ and $\operatorname{Re}\left(  \Gamma_{-n-1}\right)  )$.
Therefore, by Proposition 1, $q$ is finite-zone potential. Now suppose that
the total number of the SR periodic and antiperiodic eigenvalues is infinite.
Then by $(a)$ the number of gaps in the spectrum is also infinite
\end{proof}

To construct the finite-zone potentials we use the following consequence of
Theorem 3.

\begin{corollary}
Suppose that (2) holds. If there exists $m>0$ such that $\lambda_{n}(0)$ and
$\lambda_{n}(\pi)$ for $n>m$ are nonreal numbers then there exists $R$ such
that $[R,\infty)\subset\sigma(L(q))$ and the number of gaps in the real part
$\operatorname{Re}(\sigma(L(q)))$ of $\sigma(L(q))$ is finite.
\end{corollary}

The proof follows from Theorem 3. Nevertheless, below we give the second and
more simple proof by using the following: $F(\lambda)\in\mathbb{R}$ for all
$\lambda\in\mathbb{R}$ if (2) holds (see [11]), where $F(\lambda
):=\varphi^{\prime}(1,\lambda)+\theta(1,\lambda)$ is the Hill discriminant,
$\theta$ and $\varphi$ are the solutions of
\[
-y^{^{\prime\prime}}(x)+q(x)y(x)=\lambda y(x)
\]
satisfying the initial conditions $\theta(0,\lambda)=\varphi^{\prime
}(0,\lambda)=1,$ $\theta^{\prime}(0,\lambda)=\varphi(0,\lambda)=0.$ More
precisely,we prove the following, more general statement, which also yields
Corollary 1.

\begin{theorem}
Let $q$ be locally integrable periodic potential such that%
\begin{equation}
\exists M:F(\lambda)\in\mathbb{R},\text{ }\forall\lambda\in\lbrack
M,\infty)\text{.}%
\end{equation}
If there exists $m>0$ such that $\lambda_{n}(0)$ and $\lambda_{n}(\pi)$ for
$n>m$ are nonreal numbers then there exists $R$ such that $[R,\infty
)\subset\sigma(L(q)).$
\end{theorem}

\begin{proof}
It is well-known that $\lambda_{n}(0)$ and $\lambda_{n}(\pi)$ are the zeros of
$F(\lambda)=2$ and $F(\lambda)=-2$ respectively, the asymptotic formula
\begin{equation}
F(\lambda)=2\cos\sqrt{\lambda}+O(1/\sqrt{\lambda})
\end{equation}
as $\lambda\rightarrow\infty$ holds and $\lambda\in\sigma(L(q))$ if and only
if $F(\lambda)\in\lbrack-2,2]$ (see [7]). Thus it is enough to show that there
exists a number $R$ such that $F(\lambda)\in\lbrack-2,2]$ for all $\lambda\geq
R.$ Suppose to the contrary that for any large positive number $R>M$ there
exists $\lambda_{1}\geq R$ such that $F(\lambda_{1})\notin\lbrack-2,2].$
Without loss of generality, suppose that $F(\lambda_{1})>2.$ On the other
hand, by (15), there exists $\lambda_{2}>\lambda_{1}$ such that $F(\lambda
_{2})<2.$ Since $F$ is a continuous (see [7]) and real-valued function on
$\left[  \lambda_{1},\lambda_{2}\right]  $ (see (14))) there exists
$\lambda\in\left[  \lambda_{1},\lambda_{2}\right]  $ such that $F(\lambda)=2,$
that is, $\lambda=\lambda_{n}(0)$ and hence $\lambda_{n}(0)\in\mathbb{R}$. It
contradicts the condition of the theorem
\end{proof}

\begin{remark}
There exist a lot of asymptotic formulas for $\lambda_{n}(0)$ and $\lambda
_{n}(\pi).$ Using those formulas one can find the conditions on $q$ such that
$\lambda_{n}(0)$ and $\lambda_{n}(\pi)$ are nonreal numbers which implies that
$q$ is a finite-zone potential. Suppose that we have the formulas
\[
\lambda_{n}(0)=a_{n}(1+o(1)),\text{ }\lambda_{n}(\pi)=b_{n}(1+o(1)).
\]
If there exists $m>0$ such that$\ \operatorname{Im}a_{n}\neq0$ and
$\operatorname{Im}b_{n}\neq0$ for all $n\geq m,$ then $q$ is a finite-zone
potential. Besides, it is well-known the asymptotic formulas as
\[
\lambda_{n}(0)-\lambda_{-n}(0)=c_{n}(1+o(1)),\text{ }\lambda_{n}(\pi
)-\lambda_{-n-1}(\pi)=d_{n}(1+o(1)).
\]
It follows from (3), (5) and (6) that if $\lambda_{n}(0)$ ($\lambda_{n}(\pi)$)
is a nonreal number, then $\lambda_{n}(0)-$ $\lambda_{-n}(0)=\pm
2\operatorname{Im}\lambda_{n}(0)$ ($\lambda_{n}(\pi)-\lambda_{-n-1}(\pi
)=\pm2\operatorname{Im}\lambda_{n}(\pi)$). Moreover, $\lambda_{n}(0)$
($\lambda_{n}(\pi)$) is a nonreal numbers if and only if $\lambda_{n}(0)-$
$\lambda_{-n}(0)$ ($\lambda_{n}(\pi)-\lambda_{-n-1}(\pi)$) is nonreal. Thus,
if there exists $m>0$ such that$\ \operatorname{Im}c_{n}\neq0$ and
$\operatorname{Im}d_{n}\neq0$ for all $n\geq m,$ then $q$ is a finite-zone
potential. Using these asymptotic formulas one can construct a large class of
the finite-zone potential. Below we construct some of them.
\end{remark}

Now we find the conditions on the \ Fourier coefficients%
\begin{equation}
q_{n}=\int_{0}^{1}q(x)e^{-i2\pi nx}dx,\text{ }f_{n}=\int_{0}^{1}f(x)\cos2\pi
nxdx,\text{ }g_{n}=\int_{0}^{1}g(x)\sin2\pi nxdx\text{ }%
\end{equation}
for which $q$ is finite-zone potential, where $f=\operatorname{Re}q$ and
$g=\operatorname{Im}q.$ It is clear that if $q$ is PT-symmetric, then $f$ and
$g$ are even and odd functions respectively and hence
\begin{equation}
q_{n}=f_{n}+g_{n},\text{ }q_{-n}=f_{n}-g_{n}.
\end{equation}
Let $Q_{n}\ $and$\,S_{n}$ be the Fourier coefficients%
\[
Q_{n}=\int_{0}^{1}Q(x)e^{-i2\pi nx}dx,\text{ }S_{n}=\int_{0}^{1}S(x)e^{-i2\pi
nx}dx,
\]
of $Q$ and $S$ defined by
\begin{equation}
Q(x)=\int_{0}^{x}q(t)\,dt,\quad S(x)=Q^{2}(x)
\end{equation}
and%
\begin{equation}
P_{n}=q_{n}q_{-n}-q_{n}\left(  S_{-n}-2Q_{0}Q_{-n}\right)  -q_{-n}%
(S_{n}-2Q_{0}Q_{n}).
\end{equation}
Using (34) and (35) of [13] one can formulate Theorem 8 of [13] as follows.

\begin{summary}
Suppose that $q$ satisfies (2) and $q\in W_{1}^{s}[0,1]$ for some $s\geq0$ and
there exists positive constants $m$ and $\alpha$ such that for $n>m$ the
inequality
\[
\left\vert P_{n}\right\vert >\alpha n^{-2s-2}%
\]
holds. Then $\lambda_{n}(0)$ $\ $and $\lambda_{n}(\pi)$ are the real numbers
if and only if $P_{2n}\geq0$ and $P_{2n+1}\geq0$ respectively.
\end{summary}

Therefore from Corollary 1 we obtain the following.

\begin{theorem}
Suppose that $q$ satisfies (2) and $q\in W_{1}^{s}[0,1]$ for some $s\geq0.$ If
there exist positive constants $m$ and $\alpha$ such that for $n>m$ the
inequality
\begin{equation}
P_{n}<-\alpha n^{-2s-2}%
\end{equation}
holds, then $q$ is a finite-zone potential.
\end{theorem}

Now we prove the other and more applicable theorem for the finite-zone potentials.

\begin{theorem}
Suppose that (2) holds and $q\in W_{1}^{s}[0,1]$ for some $s\geq0.$ If there
exist $\delta>1$, $\beta>0$\ and $m>0$ such that
\begin{equation}
\left\vert g_{n}\right\vert >\beta n^{-s-1},\text{ }\left\vert g_{n}%
\right\vert >\delta\left\vert f_{n}\right\vert
\end{equation}
\ for all $n>m,$ then $q$ is a finite-zone potential, where $f_{n}$ and
$g_{n}$ are the Fourier coefficients of $\operatorname{Re}q$ and
$\operatorname{Im}q$ defined in (16).
\end{theorem}

\begin{proof}
It readily follows from (17) and (21) that%
\[
\left\vert q_{n}\right\vert >\varepsilon n^{-s-1},\text{ \ }\left\vert
q_{-n}\right\vert >\varepsilon n^{-s-1}%
\]
and hence
\begin{equation}
-q_{n}q_{-n}=\left(  g_{n}^{2}-f_{n}^{2}\right)  >\gamma n^{-2s-2},\text{
}\frac{\left\vert q_{\pm n}\right\vert }{\left\vert q_{n}q_{-n}\right\vert
}\text{ }=O(n^{s+1})
\end{equation}
for some $\varepsilon>0$ and $\gamma>0.$ Using (2) and taking into account
that $q^{(s)}$ is integrable on $[0,1]$ one can readily verify that the
functions $Q$ and $S$ defined in (18) are $1$ periodic functions, $Q^{(s+1)}$
and $S^{(s+1)}$ are integrable on $[0,1]$. Therefore
\begin{equation}
\text{ }S_{\pm n}=o\left(  n^{-s-1}\right)  ,\text{ }Q_{\pm n}=o\left(
n^{-s-1}\right)  .
\end{equation}
Now using (22) and (23) in (19) we obtain%
\[
P_{n}=\left(  f_{n}^{2}-g_{n}^{2}\right)  (1+o(1)).
\]
Therefore if (21) holds then (20) holds too. Hence the proof follows from
Theorem 5
\end{proof}

Theorem 6 shows that one can construct the large and easily checkable classes
of the finite-zone PT-symmetric periodic potentials, by constructing the
potentials $q$ so that the Fourier coefficients of $g=\operatorname{Im}q$ is
greater (by absolute value) than the Fourier coefficient of
$f=\operatorname{Re}q.$ One of them are given in the next theorem.

\begin{theorem}
Suppose that (2) holds, $g\in W_{1}^{s+1}(a,a+1)\cap W_{1}^{s}[0,1]$ for some
$s\geq0$ and $a\in\lbrack0,1)$ and $g^{(s)}$ has a jump discontinuity at $a$
with size of the jump $c:=g^{(s)}(a+0)-g^{(s)}(a-0),$ while either $f\in
W_{1}^{s+1}[0,1]$ or $f\in W_{1}^{s+1}(b,b+1)\cap W_{1}^{s}[0,1]$ for some
$b\in\lbrack0,1)$ and $f^{(s)}$ has a jump discontinuity at $b$ with size of
the jump $d:=f^{(s)}(b+0)-f^{(s)}(b-0).$ If $\left\vert d\right\vert
<\left\vert c\right\vert $ then $q$ is a finite-zone potential.
\end{theorem}

\begin{proof}
The Fourier coefficient $g_{n}$ defined in (16) can be calculated by%
\[
g_{n}=\lim_{\varepsilon\rightarrow0}\int_{a+\varepsilon}^{a+1-\varepsilon
}g(x)\sin2\pi nxdx.
\]
Applying $s+1$ times the integration by parts formula we get
\[
\left\vert g_{n}\right\vert =\left\vert c\right\vert (2\pi n)^{-s-1}%
+o(n^{-s-1}).
\]
In the same way we obtain that either $f_{n}=o(n^{-s-1})$ or
\[
\left\vert f_{n}\right\vert =\left\vert d\right\vert (2\pi n)^{-s-1}%
+o(n^{-s-1}).
\]
Therefore if $\left\vert d\right\vert <\left\vert c\right\vert $, then (21)
holds and hence the proof follows from Theorem 6
\end{proof}


\begin{thebibliography}{99}                                                                                               %


\bibitem {}Z. Ahmed , Energy band structure due to a complex, periodic, PT
-invariant potential \textit{Phys. Lett. A} \textbf{286} (2001) 231--235.

\bibitem {}F. Bagarello, \textit{Non-Selfadjoint Operators in Quantum Physics:
Mathematical Aspects}, eds. F. Bagarello, J.-P. Gazeau, F. H. Szafraniec and
M. Znojil (John Wiley \& Sons, Inc., 2015).

\bibitem {}C. M. Bender, G. V. Dunne and P. N. Meisinger, Complex periodic
potentials with real band spectra, \textit{Phys. Lett.} A \textbf{252} (1999) 272--276.

\bibitem {}E. Caliceti and S. Graffi, Reality and non-reality of the spectrum
of PT -symmetric operators: Operator-theoretic criteria, \textit{Pramana J. of
Phys.} \textbf{73} (2009) 241-249.

\bibitem {}J. M. Cervero and A. Rodriguez , The band spectrum of periodic
potentials with PT -symmetry\textit{ J. Phys. A}: \textit{Math. Gen.}
\textbf{37} (2004) 10167-10177.

\bibitem {}K. G. Makris, R. El-Ganainy, D. N. Christodoulides and Z. H.
Musslimani, PT-Symmetric Periodic Optical Potentials, \textit{Int. J. Theor.
Phys}. \textbf{50} (2011) 1019--1041.

\bibitem {}V. A. Marchenko, \textit{Sturm-Liouville Operators and
Applications}\ (Birkhauser Verlag,

Basel, 1986).

\bibitem {}D. C. McGarvey, Differential operators with periodic coefficients
in $L_{p}(-\infty,\infty)$, \textit{J. Math. Anal. Appl.} \textbf{11} (1965) 564-596.

\bibitem {}A. Mostafazadeh, Psevdo-hermitian representation of quantum
mechanics, \textit{Int. J. of Geom. Methods Mod. Phys.} \textbf{11} (2010) 1191-1306.

\bibitem {}F. S. Rofe-Beketov, The spectrum of nonselfadjoint differential
operators with periodic coefficients, \textit{Soviet Math. Dokl.} \textbf{4}
(1963) 1563-1566.

\bibitem {}K. C. Shin, On the shape of spectra for non-self-adjoint periodic
Schr\"{o}dinger operators, \textit{J. Phys. A: Math. Gen.} \textbf{37} (2004) 8287-8291.

\bibitem {}O. A. Veliev, On the spectral singularities and spectrality of the
Hill's Operator, \textit{Operat. Matrices} \textbf{10} (2016) 57-71.

\bibitem {}O. A. Veliev, On the spectral properties of the Schrodinger
operator with a periodic PT-symmetric potential, \textit{Int. J. of Geom.
Methods Mod. Phys.}\textbf{14} (2017) 1750065.
\end{thebibliography}
\end{document}